\documentclass{amsart}

\usepackage{amsthm,amssymb,amsmath,a4wide,graphicx,tikz}
\usepackage[english]{babel}
\usepackage{subfigure}
\usepackage{bm,url}
\usepackage[normalem]{ulem}
\usepackage[percent]{overpic}
\usepackage[hidelinks]{hyperref}

\renewcommand{\P}{\mathbb P}

\title{Pseudo-random number generation with $\beta$-encoders}
\author[Charlene Kalle]{Charlene Kalle}
\address[Charlene Kalle]{Mathematisch Instituut, Leiden University, Niels Bohrweg 1, 2333CA Leiden, The Netherlands}
\email[Charlene Kalle]{kallecccj@math.leidenuniv.nl}
\author[Evgeny Verbitskiy]{Evgeny Verbitskiy}
\address[Evgeny Verbitskiy]{Mathematisch Instituut, Leiden University, Niels Bohrweg 1, 2333CA Leiden, The Netherlands\textit{ and }}
\address{
Bernoulli Institute for Mathematics, Computer Science and Artificial Intelligence,
University of Groningen, 
PO Box 407
9700 AK Groningen, The Netherlands }

\email[Evgeny Verbitskiy]{evgeny@math.leidenuniv.nl}
\author[Benthen Zeegers]{Benthen Zeegers$^\dagger$}
\address[Benthen Zeegers]{Mathematisch Instituut, Leiden University, Niels Bohrweg 1, 2333CA Leiden, The Netherlands}
\email[Benthen Zeegers]{benthen@math.leidenuniv.nl}

\begin{document}


\begin{abstract}
The $\beta$-encoder is an analog circuit that converts an input signal $x \in [0,1]$ into a finite bit stream $\{b_i\}$. The bits $\{b_i\}$ are correlated and therefore are not immediately suitable for random number generation, but they can be used to generate bits $\{a_i\}$ that are (nearly) uniformly distributed. In this article we study two such methods. In the first part the bits $\{a_i\}$ are defined as the digits of the base-2 representation of the original input $x$. Under the assumption that there is no noise in the amplifier we then study a question posed by Jitsumatsu and Matsumura on how many bits $b_1, \ldots, b_m$ are needed to correctly determine the first $n$ bits $a_1,\ldots,a_n$. In the second part we show this method fails for random amplification factors. Nevertheless, even in this case,
nearly uniformly distributed bits can still be generated from $b_1,\ldots,b_m$ using modern cryptographic techniques.
\end{abstract}

\newtheorem{prop}{Proposition}[section]
\newtheorem{theorem}{Theorem}[section]
\newtheorem{lemma}{Lemma}[section]
\newtheorem{cor}{Corollary}[section]
\newtheorem{remark}{Remark}[section]
\theoremstyle{definition}
\newtheorem{defn}{Definition}[section]
\newtheorem{ex}{Example}[section]

\maketitle

\section{Introduction}

Any real number $x\in[0,1]$ can be represented in base 2 as
\begin{equation}\label{binreps}
x=\sum_{n=1}^\infty \frac {a_n}{2^n}, \quad a_{n}\in\{0,1\}.
\end{equation}
With the exception of a countable set of dyadic rationals of the form 
$x=\frac {K}{2^N}$, $K,N\in\mathbb Z_+$, the representation (\ref{binreps}) is unique. The digits $\{a_n=a_n(x)\}$ can be obtained iteratively as follows: let $x_0=x$, and for $n\ge 1$, we let
\begin{equation}\label{eq:binexpanster}
a_n=\begin{cases}0,\text{ if \/ } 2x_{n-1}<1,\\
1,\text{ if \/ } 2x_{n-1}\ge 1,
\end{cases}\text{ and}\quad x_n=2x_{n-1}-a_n.
\end{equation}
%Moreover, if $x$ is chosen uniformly in $[0,1]$, the corresponding bits $\{a_n(x)\}$ are independent
%identically distributed random Bernoulli random variables.

\medskip
Similarly, for $\beta\in (1,2)$, any number $x\in[0,1]$ can also be represented in \emph{non-integer base} $\beta$ as 
\begin{equation}\label{betaexpans}
x = \sum_{n=1}^\infty \frac {b_n}{\beta^n},\quad 
\end{equation}
again with binary digits $b_n$  in $\{0,1\}$. (In fact, any number $x \in \big[0, \frac1{\beta-1}\big]$ has an expansion of the form \eqref{betaexpans}.) Since $\beta \in (1,2)$ is not an integer, Lebesgue almost all points $x$ have uncountably many different $\beta$-expansions \cite{ejk,si}. This somewhat curious fact from number theory has some interesting applications in signal processing. As is well known, for each $u \in [1,\frac{1}{\beta-1}]$ expansions of the form in \eqref{betaexpans} can be obtained in a similar fashion as the base 2 expansions by setting $x_0 = x$, and for $n \geq 1$,
\begin{equation}\label{eq:binexpanster}
b_n=\begin{cases}0,\text{ if \/ } \beta x_{n-1}<u,\\
1,\text{ if \/ } \beta x_{n-1}\ge u,
\end{cases}\text{ and}\quad x_n=\beta x_{n-1}-b_n.
\end{equation}
This iteration scheme is used in $\beta$-encoders, which were introduced in \cite{DDGV02} by Daubechies et al.~in 2002. Using an \emph{amplifier} with amplification factor $\beta$ and a \emph{quantiser}
$$
Q_u(y) =\begin{cases} 0,&\text{ if } y < u,\\
1, &\text{ if }y\ge u, \end{cases}
$$
for an input signal $x = x_0$ in $[0,1]$ a $\beta$-encoder outputs bits $b_n = Q_u(\beta x_{n-1})$ where $x_n = \beta x - Q_u(\beta x_{n-1})$, see Figure \ref{ch1fig1}, which corresponds to the iteration scheme in \eqref{eq:binexpanster}. In practice, however, due to the intrinsic presence of noise in analogue circuits, the amplification factor $\beta$ and the threshold value $u$ fluctuate during the operation of a $\beta$-encoder circuit. If we denote by $(\beta_n)_{n \ge 1}$ and $(u_n)_{n \ge 1}$  the consecutive (random) amplification factors $\beta_n$ and threshold values $u_n$, respectively, used at each time step of the approximation algorithm, the $\beta$-encoder in reality outputs bits $b_n = Q_{u_n}(\beta_n x_{n-1})$ where $x_n = \beta_n x - Q_{u_n}(\beta_n x_{n-1})$. The robustness of the $\beta$-encoder in the A/D-conversion process has been studied in e.g.~\cite{DDGV,DY06,ward,JW09,DGWY,KHTA,KHA,SKMAH,Makino}.

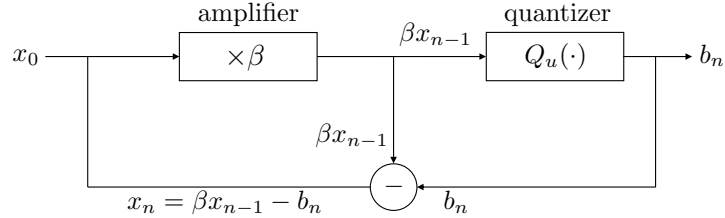
\begin{figure}[h]
\centering
\begin{overpic}[width=0.6\textwidth,tics=10]{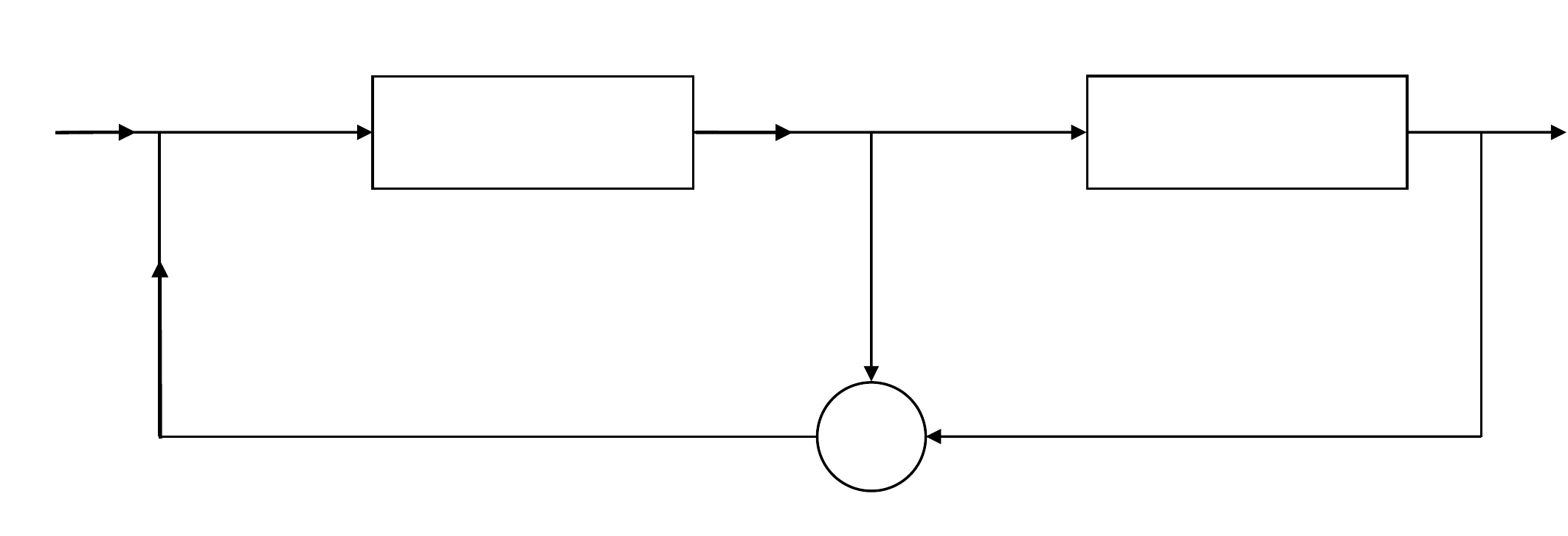}
 \put (-1,25) {$x_0$}
 \put (26.5,31) {amplifier}
 \put (72,31) {quantizer}
 \put (30.5,24.5) {\large$\displaystyle\times\beta$}
 \put (56.4,28) {$\displaystyle\beta x_{n-1}$}
 \put (101,24.8) {$b_n$}
 \put (63,3) {$b_n$}
 \put (16,3) {$x_n = \beta x_{n-1} - b_n$}
 \put (44,13) {$\displaystyle\beta x_{n-1}$}
 \put (75,24.7) {$Q_u(\cdot)$}
 \put (54.1,5.4) {$-$}
\end{overpic}
\caption{Iteration process of the $\beta$-encoder.}
\label{ch1fig1}
\end{figure}

\medskip
In recent years $\beta$-encoders were also considered as sources for random number generation, see \cite{JMKA,SJO,JM,KJ16}. If $x$ is chosen uniformly at random in $[0,1]$, then the digits $\{a_n(x)\}_{n\ge 1}$ from \eqref{binreps} form a sequence of binary independent identically distributed random variables with $\mathbb P(a_n=0)=\mathbb P(a_n=1)=\frac 12$. On the other hand, it is known that successive bits $\{b_n\}$ in the output of a $\beta$-encoder are  correlated and therefore not immediately applicable as pseudo-random numbers. Under the assumption that the amplification factor $\beta$ does not fluctuate, Jitsumatsu and Matsumura proposed in \cite{JM} a coding scheme which `removes' the dependence between the bits and converts the output bits $\{b_n\}$ of the $\beta$-encoder into the binary digits $\{a_n\}$ in base 2 of the number it represents. It is verified in \cite{JM} that the resulting output sequences $\{a_n\}$ pass the NIST statistical test suite from \cite{rukhlin01}, which shows that this method performs well as a pseudo-random number generator. A natural question asked in \cite{JM} is the following: If we use $\bm u = (u_n)_{n \ge 1}$ to denote the consecutive (random) threshold values $u_n$, what is the number $k(m,\bm u, x)$ of bits $\{b_n\}$ from the $\beta$-encoder that are necessary to obtain $m$ digits in base 2 of the number $x$ via this process? In \cite{JM} the lower bound $k(m,\bm u, x) \ge \frac{m\log 2}{\log \beta}$ was found.\footnote{This bound was found in \cite{JM} for bits $\{b_n\}$ from a \emph{scale-adjusted} $\beta$-encoder, that is, if the iteration scheme is given by \eqref{eq:binexpanster} but with $u \in [\beta-1,1]$ and $x_n = \beta x_{n-1} - (\beta-1)b_n$. This difference is not principal in the first three sections where the amplification factor is assumed to be fixed. However, in reality the amplifier and scale-adjuster are subject to noise as well, and to minimize this influence we therefore consider a model without scale-adjuster.} The authors of \cite{JM} remarked that a theoretical analysis of the expected value of $k(m,\bm u, x)$ is relevant as an indication of the efficiency of the proposed pseudo-random number generator.

\medskip
The question from \cite{JM} is reminiscent of the considerations of Lochs in \cite{lochs} from 1964, where Lochs asked how many regular continued fraction digits of a real number $x$ one can determine from knowing only the first $n$ decimal digits of $x$. If we call this number of digits $m_L(n,x)$, then Lochs' Theorem states that for Lebesgue almost every $x \in [0,1]$,
\begin{align}\label{q:lochs}
\lim_{n \rightarrow \infty} \frac{m_L(n,x)}{n} = \frac{6 \log 2 \log 10}{\pi^2}.
\end{align}
%%% m binary digits, from k= m(\log 2/\log \beta) \beta digits %%%
%% In the end, we have m digits, entropy m\log \beta, we will get nearly uniform
%% with n=(1-a) m\log \beta =K -->. m = 1/(1-a) K /\log\beta digits necessary.
The somewhat mysterious expression on the right-hand side turns out to be a ratio of \emph{entropies} of the interval maps $T(x)=10x\bmod 1$ and $S(x)= 1/x \bmod 1$ that generate the decimal expansions and regular continued fraction expansions, respectively. Lochs' result was extended in \cite{dajani01} to other types of number expansions including binary expansions and $\beta$-expansions by placing it in a dynamical systems framework, see also \cite{bosma99}. These results are further generalized in \cite{randomlochs22} to number expansions generated by random dynamical systems. Unfortunately the results from \cite{lochs,bosma99,dajani01,randomlochs22} do not immediately apply to the question from \cite{JM} due to the uncertainty in the threshold value $u$. In this article we adapt the methods from \cite{dajani01,randomlochs22} to the specific iteration scheme of the $\beta$-encoder.

\medskip
The first goal of this article is to address the question posed in \cite{JM}. In our first main result we recover the lower bound from \cite{JM} and we obtain a statement on an upper bound for $k(m,\bm u, x)$. More precisely, we obtain the following results. Here $\lambda$ denotes the one-dimensional Lebesgue measure.

\begin{theorem}\label{t:main}
Consider  $\beta \in (1,2)$ and a sequence of thresholds $\bm u = (u_n)_{n \ge 1} \in [1,(\beta-1)^{-1}] ^\mathbb N$. For all $x \in [0,1]$ and all $m \in \mathbb{N}$ it holds that
\begin{equation}\label{q:lowerboundk}
k(m,\bm u,x) \ge \frac{m\log 2}{\log \beta}.
\end{equation}
Moreover, for each $\varepsilon \in (0,1)$ there exists a constant $C(\varepsilon) > 0$ such that for all $m \in \mathbb{N}$
\begin{equation}\label{q:upperboundk}
\lambda\Big( \Big\{ x  \in [0,1] \, : \, k(m,\bm u ,x)  - \frac{m \log 2}{\log \beta} > C(\varepsilon) \Big\} \Big) < \varepsilon.
\end{equation}
\end{theorem}

From these bounds we obtain the following corollary on the asymptotic behaviour of the sequences $(k(m,\bm u, x))_{m \ge 1}$. 

\begin{cor}\label{c:main}
For any real positive sequence $(n_m)_{m \in \mathbb{N}}$ with $\lim_{m \rightarrow \infty} n_m = \infty$, each $\bm u \in [1,(\beta - 1)^{-1}]^\mathbb N$ and $\varepsilon > 0$ it holds that
\[ \lim_{m \rightarrow \infty} \lambda\Big( \Big\{ x \in [0,1] \, :\, \frac{1}{n_m}\Big|k(m,\bm u,x)- \frac{m \log 2}{\log \beta}\Big| > \varepsilon \Big\}\Big) = 0,\]
i.e., the sequence $ \big( \frac1{n_m} \big(k(m,\bm u,x)- \frac{m\log 2}{\log \beta}\big) \big)_{m\ge 1}$ converges to 0 in $\lambda$-probability.
\end{cor}

\noindent In particular, the above corollary has the following implications:\\
\begin{itemize}
\item Taking $n_m = \sqrt m$ for each $m$ gives a Central Limit Theorem result where the limiting distribution has zero variance;\\
\item Taking $n_m = m$ for each $m$ we retrieve a limit statement in the spirit of \eqref{q:lochs}, but with convergence in probability instead of almost surely.
\end{itemize}
\medskip
By adjusting the setup from \cite{dajani01} to suit our purposes, we obtain the stronger result of almost sure convergence for the specific sequence $(n_m)_{m \ge 1}$ with $n_m=m$ for each $m$ that is stated in the next theorem.
\begin{theorem}\label{t:main2}
For each $\bm u \in [ 1, (\beta-1)^{-1}]^\mathbb N$, it holds that
\[ \lim_{m \to \infty} \frac{k(m, \bm u, x)}{m} = \frac{\log 2}{\log \beta} \qquad \text{for $\lambda$-a.e.~$x \in [0,1]$.}\]
\end{theorem}

More specifically, for typical $x$ and large $N$ one needs approximately $N\frac{ \log 2}{\log \beta}$ output bits of the $\beta$-encoder to obtain $N$ correct binary digits.

\medskip
Since the implementation of $\beta$-encoders it has been observed that (like for the threshold value $u$) there is uncertainty about the precise value of $\beta$ during the encoding process. The actual value of $\beta$ can only be determined to lie within an interval $[ \beta_{\min}, \beta_{\max} ]$. Possible solutions to this problem were studied in \cite{DGWY,DY06,ward}. We will argue that in this case one is not able to extract a large number of digits $(a_1,\ldots,a_n)$ in the base 2 expansion of the input value $x$ using the output bits $(b_1,\ldots,b_m)$ from the $\beta$-encoder. Nevertheless, the output bits $(b_1,\ldots,b_m)$ are still sufficiently random, and using modern cryptographic techniques, one is still able to extract $n$ nearly independent bits from $\gamma n\frac{\log 2}{\log \beta}$ output bits, where $\gamma>1$ is a fixed factor, which depends on how close to `nearly independent' the final output bits should be.

\medskip
The article is organised as follows. In the next section we introduce the necessary notation and preliminaries on base 2 expansions and $\beta$-expansions. In Section \ref{sec3} we prove Theorem \ref{t:main}, Corollary \ref{c:main} and Theorem \ref{t:main2}. Here it is assumed that the amplification factor is fixed and only the threshold value fluctuates. Finally, in Section \ref{sec4} we discuss modern cryptographic techniques to apply for the case that the amplification factor fluctuates as well.

\section{Preliminaries}\label{s:prereq}

For a set $A$ and an integer $m \ge 1$ we use the notation $A^m = \{ (a_1, \ldots, a_m) \, : \, a_i \in A, \, 1 \le i \le m\}$ and $A^\mathbb N = \{ (a_k)_{k \ge 1} \, : \, a_k \in A, \, k \ge 1 \}$. If $I$ is an interval in the real line, then we write $\partial I$ for the set containing the two boundary points of $I$ and we use $I^-$ and $I^+$ to denote the left and right endpoints of $I$, respectively.

\medskip
For each $m \geq 1$ the collection of {\em dyadic intervals of order $m$} is given by
\[ \mathcal D_m = \Big\{ \Big[ \frac{k}{2^m}, \frac{k+1}{2^m} \Big) \, : \, 0 \le k \le 2^m-1 \Big\}.
\]
If we write the point $\frac{k}{2^m} = \sum_{i=1}^m \frac{d_i}{2^i}$, $d_i \in \{0,1\}$, in its binary expansion, then we see that the interval $\big[ \frac{k}{2^m}, \frac{k+1}{2^m} \big)$ contains precisely those $x \in [0,1)$ that have $d_1, \ldots, d_m$ as their first $m$ binary digits. For each $x\in [0,1)$ and each $m \ge 1$ there is a unique element of $\mathcal D_m$ that contains $x$. We denote this interval by $\mathcal D_m(x)$. Then
\begin{equation}\label{q:sized}
\lambda (\mathcal D_m(x)) = 2^{-m}.
\end{equation}
Hence, each collection $\mathcal{D}_m$ is a partition of $[0,1)$ by intervals of length $2^{-m}$. By adding the point 1 to the last interval of $\mathcal{D}_m$ we obtain a partition of the closed interval $[0,1]$ without disturbing any of the properties mentioned above.

\medskip
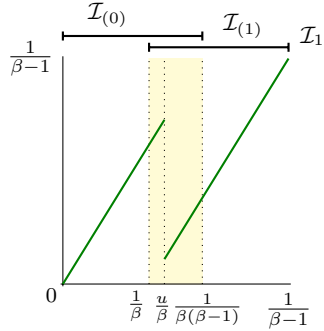
\begin{figure}[h]
\begin{tikzpicture}[scale=3]
\filldraw[fill=yellow!20, draw=yellow!20] (.3819,0) rectangle (.618,1);
\draw[dotted] (.3819,0)--(.3819,1)(.618,0)--(.618,1);
\draw (-.01,0)--(.325,0)node[below=-1pt]{\footnotesize $\frac{1}{\beta}$}--(.44,0)node[below=1pt]{\footnotesize $\frac{u}{\beta}$}--(.64,0)node[below]{\footnotesize $\frac{1}{\beta(\beta-1)}$}--(1,0)node[below]{$\footnotesize \frac{1}{\beta-1}$}--(1.01,0)(0,-.01)--(0,1)node[left]{$\footnotesize \frac{1}{\beta-1}$}--(0,1.01);
\node at (-.05,-.05) {\footnotesize 0};
\draw[dotted] (.45,0)--(.45,1);
\draw[thick,green!50!black](0,0)--(.45,.728)(.45,.1101)--(1,1);
\draw[thick] (0,1.1)--(.618,1.1)(0,1.08)--(0,1.12)(.618,1.08)--(.618,1.12);
\node[above] at (.2,1.1) {\small $\mathcal I_{(0)}$};
\draw[thick] (.3819,1.05)--(1,1.05)(.3819,1.07)--(.3819,1.03)(1,1.07)--(1,1.03);
\node[above] at (.8,1.05) {\small $\mathcal I_{(1)}$};
\node[right=1pt] at (1,1.1) {\small $\mathcal I_1$};
\end{tikzpicture}
\caption{The graph of one of the maps $T_u$ is shown for $\beta = \frac{1+\sqrt 5}{2}$, the golden mean. The yellow area in the middle relates to the interval in which the threshold value $u$ may be chosen. At the top we see the two intervals $\mathcal I_{(0)}$ and $\mathcal I_{(1)}$ that are the elements of the cover $\mathcal I_1$.}
\label{f:Tu}
\end{figure}

\medskip
Usually A/D-converters rely on binary expansions of numbers to produce good approximations of the input signal. The $\beta$-encoder is based on $\beta$-expansions instead. Fix a value of $\beta \in (1,2)$. An expression of the form
\[ x = \sum_{n \ge 1} \frac{b_n}{\beta^n}, \quad b_n \in \{0,1\},\]
is called a $\beta$-expansion of $x$. The set of numbers that can be written in this way is equal to the interval $\big[ 0, \frac1{\beta-1} \big]$.
%One easily sees that if $x$ has such a $\beta$-expansion, then $x \in \big[ 0, \frac1{\beta-1} \big]$. The $\beta$-encoder as described in \cite{JM} considers as input signal a number $x \in [0,1]$ and thus has rescaled the setup by a factor $\beta-1$.
We now briefly explain how one can get a $\beta$-expansion of a number $x$
 from the $\beta$-encoder introduced in the introduction with varying threshold values $u_n$.

\medskip
For each $u \in [1,(\beta - 1)^{-1}]$ define the interval map $T_u: [0,(\beta - 1)^{-1}] \rightarrow [0,(\beta - 1)^{-1}]$ by
\begin{equation}\label{q:defTu}
T_u(y) = \begin{cases}
\beta y, & \text{if } y < \frac{u}{\beta},\\
\beta y - 1, & \text{if }y \ge \frac{u}{\beta}.
\end{cases}
\end{equation}
The graph of such a map is shown in Figure~\ref{f:Tu}. If we let $u_n$ denote the threshold value of the quantiser at time $n$, then the dynamics of the $\beta$-encoder can be represented as
\begin{equation}\label{q:betaorbit}
x_n = T_{u_n}(x_{n-1}) = T_{u_n} \circ \cdots \circ T_{u_1}(x), \quad n \geq 1.
\end{equation}
For each $n \ge 1$, set $b_n = b_n(x) = 0$ if $\beta x_{n-1} < u_n$ and 1 otherwise. Putting $x_0 = x$, then for each $n \ge 1$,
\[ T_{u_n} (x_{n-1}) = \beta x_{n-1} - b_n,\]
so that
\[ x = \sum_{i=1}^n \frac{b_i}{\beta^i} + \frac{T_{u_n} \circ \cdots \circ T_{u_1}(x)}{\beta^n}.\]
Since $T_{u_n} \circ \cdots \circ T_{u_1}(x) \in [0,(\beta-1)^{-1}]$ holds for each $n$, we immediately conclude that   $x = \sum_{n=1}^\infty \frac{b_n}{\beta^n}$. From Figure~\ref{f:Tu} it becomes clear that each threshold value $u_n$ must lie in the interval $[1,(\beta - 1)^{-1}]$ to obtain a recursive process and bits that correspond to $\beta$-expansions. It follows from \cite[Theorem 2]{DV05}, where for the case that $\beta \in (1,2)$ only the choices $u_n \in \{1,(\beta-1)^{-1}\}$ for each $n \geq 1$ are considered, that in fact all $\beta$-expansions can be generated using the above iteration process.

\begin{remark}{\rm Note that if one starts this process with a number $x\in [0,1]$, then typically $x_n >1$ for many $n$. The reason to look at $x \in [0,1]$ instead of $x \in [0,\frac{1}{\beta-1}]$ is to make the comparison with the dyadic intervals $\mathcal{D}_m(x)$, which are defined on $[0,1]$, easier.} %Since we are only interested in the digit sequence $(b_n)$ corresponding to a $\beta$-expansion of $x$ and not so much in the value of the numbers $x_n$, the fact that $x_n$ can be larger than 1 poses no problem.}
\end{remark}

\medskip
%The sequence $(b_n)_{n \ge 1}$ that appears as output from a $\beta$-encoder can thus either be seen as the digit sequences corresponding to a $\beta$-expansion of the number $\frac{x}{\beta-1}$ or as a $\beta$-expansion of the input signal $x$ with digits $0$ and $\beta-1$ instead of digits 0 and 1.
Given the first $k$ output bits $b_1, \ldots, b_k$ of the $\beta$-encoder, we know that the input signal $x \in [0,1]$ has to satisfy
\[ x \in \Big[ \sum_{n=1}^k \frac{b_n}{\beta^n}, \sum_{n=1}^k \frac{b_n}{\beta^n} + \sum_{n \ge k+1} \frac1{\beta^k} \Big] =  \Big[ \sum_{n=1}^k \frac{b_n}{\beta^n}, \sum_{n=1}^k \frac{b_n}{\beta^n} + \frac1{\beta^k (\beta-1)} \Big].\]
For each $b_1, \ldots, b_k \in \{0,1\}$ define
\[ \mathcal I_{(b_1, \ldots, b_k)} = \Big[ \sum_{n=1}^k \frac{b_n}{\beta^n}, \sum_{n=1}^k \frac{b_n}{\beta^n} + \frac1{\beta^k(\beta-1)} \Big].\]
Comparable to the partitions $\mathcal D_m$ for binary expansions, we consider for each $k \ge 1$ the {\em cover} $\mathcal I_k$ of $[0,(\beta-1)^{-1}]$ associated to $\beta$-expansions given by
\[ \mathcal I_k = \{ \mathcal I_{(b_1, \ldots, b_k)} \, : \, b_i \in \{0,1\}, \, 1 \le i\le k \}.\]
See Figure~\ref{f:Tu} for an illustration of $\mathcal I_1 = \{ \mathcal I_{(0)}, \mathcal I_{(1)} \}$.

\medskip
If for $k \ge 1$ the first $k$ output bits of the $\beta$-encoder for an input signal $x \in [0,1]$ and a threshold value sequence $\bm u \in [1,(\beta -1)^{-1}]^\mathbb N$ are $b_1, \ldots, b_k$, then we set
\begin{equation*}%\label{q:interval}
 \mathcal I_k(\bm u, x)= \mathcal I_{(b_1, \ldots, b_k)},
\end{equation*}
since the information that the bits $b_1, \ldots, b_k$ give us is that $x$ is contained in this interval. Note that
\begin{equation}\label{q:sizei}
\lambda(\mathcal I_k(\bm u, x)) = \frac1{\beta^k(\beta-1)}.
\end{equation}
Furthermore, 
\begin{equation}\label{eq11}
k(m ,\bm u, x) = \inf \{ k\ge 1 \, : \, \mathcal I_k(\bm u,x) \subseteq \mathcal D_m(x)\}.
\end{equation}

\section{Fixed amplification factor}\label{sec3}

In this section we prove our first main results where the amplification factor is assumed to be fixed. We start with the proof of Theorem~\ref{t:main}, which provides bounds for the quantities $k(m, \bm u,x)$. This proof is inspired by the proof of \cite[Theorem 2.3]{atilla}.

\begin{proof}[Proof of Theorem~\ref{t:main}]
Fix $\bm u = (u_n)_{n \ge 1} \in [ 1, (\beta-1)^{-1}]^\mathbb N$. For all $m \in \mathbb N$ and $x \in [0,1]$ we find by \eqref{q:sized} and \eqref{q:sizei} that $\lambda(\mathcal D_m(x))=2^{-m}$ and $\lambda(\mathcal I_{k(m,\bm u,x)}(\bm u, x))=\beta^{-k(m,\bm u,x)}(\beta-1)^{-1}$. Hence,
\begin{equation}\label{q:kmux} \begin{split}
k(m,\bm u,x) - \frac{m \log 2}{\log \beta} + \frac{\log(\beta-1)}{\log \beta}
%=\ & \frac{k(m,\bm u,x) \log \beta + \log \lambda(\mathcal I_{k(m,\bm u,x)}(\bm u,x))}{\log \beta}\\
%& + \frac{-\log \lambda(\mathcal I_{k(m,\bm u,x)}(\bm u, x))+\log \lambda(\mathcal D_m(x))}{\log \beta}  \\
%& + \frac{-\log \lambda(\mathcal D_m(x)) -  m \log 2}{\log \beta} \\
=\ & \frac1{\log \beta} \cdot \log\bigg(\frac{\lambda(\mathcal D_m(x))}{ \lambda(\mathcal I_{k(m,\bm u,x)}(\bm u, x))}\bigg).
\end{split}\end{equation}
Furthermore, by the definition of $k(m,\bm u,x)$ we have $\mathcal I_{k(m,\bm u,x)}(\bm u, x) \subseteq \mathcal D_m(x)$ and since $\beta \in (1,2)$ the above yields
\[ k(m,\bm u,x) \ge \frac{m\log 2}{\log \beta} -\frac{\log(\beta-1)}{\log \beta} > \frac{m\log 2}{\log \beta}.\]
This gives \eqref{q:lowerboundk}. 

\medskip
For \eqref{q:upperboundk} let $\varepsilon \in (0,1)$ and fix some integer $m \ge 1$. By the definition of $k(m,\bm u,x)$ we have that $\mathcal I_{k(m,\bm u,x)-1}(\bm u, x) \nsubseteq \mathcal D_m(x)$. Hence, the distance between $x$ and the nearest boundary point of $\mathcal D_m(x)$, denoted by $|x-\partial \mathcal D_m(x)|$, is at most equal to $\lambda(\mathcal I_{k(m,\bm u,x)-1}(\bm u, x))$. Furthermore, we have
\[ \log \lambda(\mathcal I_{k(m,\bm u,x)-1}(\bm u, x)) - \log \lambda(\mathcal I_{k(m,\bm u,x)}(\bm u, x)) = \log \beta.\]
Together this gives that
\begin{equation}\label{q:DI}
\log\Big(\frac{\lambda(\mathcal D_m(x))}{ \lambda(\mathcal I_{k(m,\bm u,x)}(x))}\Big) \le \log \lambda\big(\mathcal D_m(x)\big) + \log \beta - \log |x-\partial \mathcal D_m(x)|.
\end{equation}
We slightly adjust the intervals in $\mathcal D_m$ by removing small intervals at the endpoints: For each $m \in \mathbb N$ and interval $J \in \mathcal D_m$, let $J'$ be the interval obtained by removing on both ends of $J$ an interval of length $\frac{\varepsilon}{2} \cdot 2^{-m}$ and let $C_m = \bigcup_{J \in \mathcal D_m} J'$. Then $\lambda(J') = (1-\varepsilon ) \cdot 2^{-m}$ and $\lambda(C_m) = 1 - \varepsilon$. For $x \in C_m$ we have the bound $|x- \partial \mathcal D_m(x)| \ge \frac{\varepsilon}{2}\lambda(\mathcal D_m(x))$. Combining this with \eqref{q:kmux} and \eqref{q:DI} gives for each integer $m \in \mathbb{N}$ and each $x \in C_m$ that
\[ k(m,\bm u,x) - \frac{m\log 2}{\log \beta} \le \frac{\log\frac{2}{\varepsilon}}{\log \beta} + 1.\]
Hence, we obtain \eqref{q:upperboundk} with constant $C(\varepsilon) = \frac{\log\frac{2}{\varepsilon}}{\log \beta} + 1$.
\end{proof}

Theorem~\ref{t:main} gives bounds on the value of $k(m, \bm u,x)$ and immediately leads to the  statement on the asymptotics of the sequence $(k(m, \bm u,x))_{m \ge 1}$ from Corollary~\ref{c:main} that we prove next.

\begin{proof}[Proof of Corollary~\ref{c:main}]
Let $(n_m)_{m \ge 1}$ be a sequence of positive real numbers that satisfy $\lim_{m \to \infty} n_m = \infty$. From \eqref{q:lowerboundk} we get that for each $x \in [0,1]$ and $m \in \mathbb{N}$,
\[ \frac1{n_m} \bigg( k(m, \bm u, x) - \frac{m \log 2}{\log \beta} \bigg) \ge 0.\]
Hence, it suffices to show that for all $\delta, \varepsilon > 0$ there exists an $M \in \mathbb N$ such that for all $m \ge M$ we have
\[ \lambda\Big( \Big\{ x\in [0,1] \, : \,  \frac{1}{n_m} \Big(k(m,\bm u,x)- \frac{m\log 2}{\log \beta} \Big)  > \varepsilon \Big\}\Big) < \delta.\]
This immediately follows from \eqref{q:upperboundk} by taking $M \in \mathbb{N}$ big enough such that $\frac{C(\delta)}{n_m} \le \varepsilon$ for all $m \geq M$, which is possible because $\lim_{m \rightarrow \infty} n_m = \infty$.
\end{proof}

As we saw in the introduction, by choosing $n_m=m$ for all $m \ge 1$, Corollary~\ref{c:main} gives a limit statement reminiscent of Lochs' Theorem, but with convergence in probability. Our final result, Theorem~\ref{t:main2}, shows that this limit statement also holds almost surely. The proof we present for Theorem~\ref{t:main2} below is inspired by the proof of \cite[Theorem 4]{dajani01}.

\begin{proof}[Proof of Theorem~\ref{t:main2}]
Fix some $\bm u \in [\beta - 1, 1]^\mathbb N$. It follows from \eqref{q:lowerboundk} that for all $x \in [0,1]$
\[ \liminf_{m \to \infty} \frac{k(m, \bm u,x)}{m} \ge \frac{\log 2}{\log \beta}.\]
Conversely, let $\varepsilon \in (0,1)$ and for each $m \ge 1$ define $\bar{k}(m) = \lceil (1+\varepsilon)\frac{m\log 2}{\log \beta} \rceil$. Let
\[ \begin{split}
\mathcal P_m =\ & \{ x \in [0,1] \, : \, \mathcal I_{\bar{k}(m)} (\bm u, x) \not \subseteq \mathcal D_m(x) \}\\
\subseteq\ & \bigcup_{B \in \mathcal D_m} \bigcup_{A \in \mathcal I_{\bar{k}(m)}\, :\,  A \not \subseteq B} A \cap B\\
\subseteq \ & \bigcup_{B \in \mathcal D_m} [ B^-,  B^- + \beta^{-(1+\varepsilon)\frac{m \log 2}{\log \beta}}] \cup [ B^+- \beta^{-(1+\varepsilon)\frac{m \log 2}{\log \beta}},  B^+].
\end{split} \]
Since $\mathcal D_m$ has 
$2^m=\beta^{\frac{m \log 2}{\log \beta}}$
  elements, we have
\[  \lambda\big(\mathcal P_m\big) \le \beta^{\frac{m \log 2}{\log \beta}} \cdot 2 \cdot \beta^{-(1+\varepsilon)\frac{m \log 2}{\log \beta}} \le 2 \cdot \beta^{-\varepsilon\frac{m \log 2}{\log \beta}}
 =2\cdot 2^{-\varepsilon m},\]
which gives that $\sum_{m=1}^{\infty} \lambda\big(\mathcal P_m\big) < \infty$. From the Borel-Cantelli Lemma it follows that
\[ \lambda\big(\{x \in [0,1] \, :\,  x \in \mathcal P_m \text{ for infinitely many $m \in \mathbb N$}\}\big) = 0.\]
Hence,
\[ \lambda\big(\{x \in [0,1] \, :\,  \exists M \in \mathbb{N} \text{ s.t. } \forall m \geq M  \, \, \mathcal I_{\bar{k}(m)} (\bm u, x) \subseteq \mathcal D_m(x) \}\big) = 1,\]
or in other words, for Lebesgue almost all $x \in [0,1]$ there exists an $M \in \mathbb N$ such that for all $m \ge M$ it holds that $k(m, \bm u,x) \le \bar k(m)$. This gives
\[ \limsup_{m \rightarrow \infty} \frac{k(m,\bm u,x)}{m} \le \limsup_{m \to \infty} \frac{\bar k(m)}{m} = (1+\varepsilon)\frac{\log 2}{\log \beta}, \qquad \lambda\text{-a.e.}\]
Since $\varepsilon > 0$ was arbitrary, this concludes the proof.
\end{proof}

\begin{remark}{\rm Note that the first part of the previous proof holds for all $x \in [0,1]$. It is the second part that only holds Lebesgue almost everywhere.
}\end{remark}

\section{Random amplification factor}\label{sec4}
In practice it is not only the threshold value $u$ that is subject to fluctuations present in the circuit, but also the amplification factor $\beta$. This issue and its implications 
for signal processing were discussed extensively in \cite{DY06,ward,DGWY}. 
Under some extra assumptions, e.g., amplification factors varying slowly and smoothly,
one can find some ways to remedy this issue.
However, in the general case, 
%if we
%Here we discuss the consequences for the number $n$ of binary digits one can accurately determine from the first $m$ digits of the $\beta$-expansion. 
%
%Unfortunately, 
as the following simple consideration shows, in the presence of random amplification factors, one
cannot expect to reliably determine \emph{a significant number} of digits in the base 2 expansion of the input signal $x$ by linking them
to the digits from a random $\beta$-expansion of $x$.

\medskip
Let us  start by modeling the random amplification factors. Suppose that at each iteration the amplification factor $\beta$ assumes a random value in some interval $[\beta_{\min}, \beta_{\max}] \subseteq (1,2)$. Denote by $\bm\beta = (\beta_n)_{n \ge 1} \in [\beta_{\min}, \beta_{\max}]^\mathbb N$ the corresponding sequence. Similarly, we denote by $\bm u= (u_n)_{n \ge 1}$ again the sequence of the corresponding random threshold values. We assume $u_n\in [1,(\beta_{\max}-1)^{-1}]$ for all $n$. As we will see below, the sequence $\bm u$ will not have any effect on the conclusions.

\medskip
Again, randomly choose $x_0=x$ uniformly in  $[0,1]$. The bits 
$b_n$, $n\ge 1$, are defined iteratively by
\begin{equation}\label{eq:randbetagen}
b_n=Q_{u_n}(\beta_n x_{n-1})=\begin{cases} 0,&\text{ if } \beta_nx_{n-1}<u_n\\
1,&\text{ if } \beta_nx_{n-1}\ge u_n\\
 \end{cases},
 \text{ and } x_n = \beta_n x_{n-1}-b_n.
\end{equation}
Thus for all $n$, one has
\begin{equation}\label{eq:rndbeta}
x = \sum_{i=1}^n \frac{b_i }{\prod_{j=1}^i \beta_j} + \frac{x_n}{\prod_{j=1}^n \beta_j}.
\end{equation}
\begin{lemma}
We have $x_n \leq (\beta_{\max}-1)^{-1}$ for all $n$.
\end{lemma}
\begin{proof}
We have $x_0 \leq 1 \leq (\beta_{\max}-1)^{-1}$. Now suppose $x_n \leq (\beta_{\max}-1)^{-1}$ holds for some $n$. If $b_{n+1} = 0$, then $$x_{n+1} = \beta_{n+1} x_n - b_{n+1} = \beta_{n+1} x_n < u_{n+1} \leq (\beta_{\max}-1)^{-1}.$$ On the other hand, if $b_{n+1} = 1$, then
$$x_{n+1} = \beta_{n+1} x_n - 1 \leq \frac{\beta_{n+1}}{\beta_{\max}-1} - 1 \leq \frac{\beta_{\max} - \beta_{\max} + 1}{\beta_{\max}-1} = (\beta_{\max}-1)^{-1}.$$
So the statement holds in both cases.
\end{proof}

Setting $\varkappa=(\beta_{\max}-1)^{-1}$, it follows from the above lemma and \eqref{eq:rndbeta} that
\begin{equation}\label{eq:rndbetabound}
 0\le x-\sum_{i=1}^n \frac{b_i }{\prod_{j=1}^i \beta_j}\le\frac {\varkappa}{\beta_{\min}^n}\to 0\text{ as }n\to \infty.
\end{equation}
Hence, the digits $b_n$ correspond to an expansion of $x$ of the form $x = \sum_{i=1}^\infty \frac{b_i }{\prod_{j=1}^i \beta_j}$. These are called a {\em Cantor real base expansions} and are studied in \cite{CC21}.
%{\color{red}{\bf NO LONGER RELEVANT} Writing $T_{\beta,u}$ for the map $T_u$ from \eqref{q:defTu}, it follows from the above lemma that each $x_n$ lies in the domain $[0,(\beta_n-1)^{-1}]$ of $T_{\beta_n,u_n}$. So we also have for $x_n$ the expression
%$$x_n = T_{\beta_n,u_n}(x_{n-1}) = T_{\beta_n,u_n} \circ \cdots \circ T_{\beta_1,u_1}(x)$$
%}

\medskip
However, given the first $m$ output digits $b_1, \ldots, b_m$, without exact knowledge on the  sequence $\bm \beta $ of random $\beta$-encoder amplifications, the only certain conclusion about the location of $x=x_0$ one can draw  from \eqref{eq:rndbeta} is that
$$ x \in \hat{\mathcal I}_{(b_1, \ldots, b_m)} := \left[  \sum_{k=1}^m \frac{b_k}{\beta_{\max}^k}, \sum_{k=1}^m \frac{b_k}{\beta_{\min}^k} + \frac{\varkappa}{\beta_{\min}^m} \right].
$$
The immediate conclusion is that the length $\hat{\mathcal I}_{(b_1, \ldots, b_m)}$ does not converge\footnote{Unless all $\beta$-digits $b_n$ are $0$} to $0$ as $m\to\infty$, and hence
we can not reliably determine a large number of  binary digits of $x$.
Hence, under the assumption that amplification factors fluctuate in the $\beta$-encoder circuit, one cannot guarantee the quality of the corresponding  pseudo-random number generators
studied earlier in the literature.

\medskip
Nevertheless, it is absolutely clear, that the `random' $\beta$-expansion circuit does produce digits $(b_m)$ which are sufficiently random, and hence can, in principle, be used in random number generators. The natural practical questions are how much randomness is in $(b_1,\ldots,b_m)$, and how can one \emph{extract} this randomness?

\medskip
Let us start with the first question. Suppose $\boldsymbol\beta=(\beta_n)$  is a random process of random  amplification factors assuming values $\beta_n\in [\beta_{\min},\beta_{\max}]$ for all $n$. We denote by $\rho$ the corresponding probability law on $[\beta_{\min},\beta_{\max}]^{\mathbb N}$. As we will see, the threshold values $\bm u=(u_n)$ will not be important. For convenience we will assume $u_n=1$ for all $n$. The initial point $x=x_0$ will be chosen uniformly in $[0,1]$. Recall that $\lambda$ denotes the Lebesgue measure on $[0,1]$. Let $\Omega=[\beta_{\min},\beta_{\max}]^{\mathbb N}\times [0,1]$  and let $\P=\rho\times\lambda$ denote the corresponding probability law. Consider now the first $m$ $\beta$-digits $(b_1,\ldots,b_m)$ obtained according to \eqref{eq:randbetagen}. We will view $b_1=b_1(\omega), \ldots, b_m=b_m(\omega)$ as random variables on $\Omega$ with $\omega=(\bm\beta,x_0)$ distributed according to  $\P=\rho \times\lambda$.

\medskip
One way to quantify randomness in $(b_1,\ldots,b_m)$ is to estimate the so-called {\em $\min$-entropy} ${\mathbf H}_\infty(\P_m)$ of the corresponding probability distribution $\P_m$ on the space of binary strings of length $m$. If we write $c_1^m:= c_1 \cdots c_m \in \{0,1\}^m$ for a word of length $m$, then
$$\aligned
{\mathbf H}_\infty(\P_m) := \ & \min_{c_1^m\in \{0,1\}^m} \log_2 \frac 1{\P_m(c_1^m)}=
 -\log_2 \max_{c_1^m\in \{0,1\}^m} \P_m(c_1^m)\\
 = \ & 
-\log_2 \max_{c_1^m\in \{0,1\}^m}\P\left(\left\{ \omega\in\Omega: b_1(\omega)=c_1,\ldots, b_m(\omega)=c_m\right\}\right).
\endaligned
$$
The lower bound on ${\mathbf H}_\infty(\P_m)$ is relatively straightforward: indeed, for any
$c_1^m\in \{0,1\}^m$, by the law of total probability,
$$\aligned
\P_m(c_1^m)&= \P(\{(\bm\beta,x)\in\Omega:\ b_1(\bm\beta,x)=c_1,\ldots,b_m(\bm\beta,x)=c_m\})\\
&=\int\limits_{[\beta_{\min},\beta_{\max}]^{\mathbb N}} \lambda\Bigl(\Bigl\{x\in [0,1]: b_1(\bm\beta,x)=c_1, \ldots,
b_m(\bm\beta,x)=c_m\Bigr\}\Bigr) \rho(d\bm\beta).
\endaligned
$$
For fixed $\beta_1,\ldots, \beta_m$, one has
$$
\Bigl\{x\in [0,1]: b_1(\bm\beta,x)=c_1, \ldots,
b_m(\bm\beta,x)=c_m\Bigr\}
\subseteq
\left[  \sum_{i=1}^m \frac{c_i }{\prod_{j=1}^i \beta_j},
 \sum_{i=1}^m \frac{c_i }{\prod_{j=1}^i \beta_j}+
\frac{\varkappa}{\prod_{j=1}^m \beta_j} 
\right],
$$
and hence,
$$\aligned
\P_m(c_1^m)&\le 
\int\limits_{[\beta_{\min},\beta_{\max}]^{\mathbb N}} \frac{\varkappa}{\prod_{j=1}^m \beta_j}\, \rho(d\bm\beta)\le 
\frac{\varkappa}{(\beta_{\min})^m}\int\limits_{[\beta_{\min},\beta_{\max}]^{\mathbb N}} 
\rho(d\bm\beta)=\frac{\varkappa}{(\beta_{\min})^m}.
\endaligned
$$
Therefore,
\begin{equation}\label{eq:lowerminentropy}
\mathbf{H}_\infty(\P_m)\ge m\log_2 \beta_{\min}-\log_2 \varkappa.
\end{equation}

\vskip .2cm
This argument shows that the min-entropy of our physical source of randomness -- the $\beta$-encoder circuit -- grows linearly in $m$, and that the growth-rate is at least $\frac{\log\beta_{\min}}{\log 2}$, i.e., the entropy of the `worst' or the least random $\beta$-transformation, which is present in the mix. A source $X$ is called a {\em random $(m,k)$-source} if $X$ takes values in $\{0,1\}^m$ and ${\mathbf H}_\infty(X)\ge k$. The computation above shows that the string of the first $m$ bits of the $\beta$-encoder $\boldsymbol b_m=(b_1,\ldots,b_m)$ is an $(m,k)$-source for any $k\le m\frac {\log\beta_{\min}}{\log 2}-\log_2 \varkappa$.

\medskip
For the next step we turn to the the theory of randomness extracts developed by 1980's by
Chor, Goldreich, Cohen, Wigderson, Zuckerman and many others (c.f., \cite{Tre,Wig}). The basic idea is, given a sufficiently  random binary vector of length $m$,  $X\in\{0,1\}^m$, find a possibly smaller integer $n$, $n\le m$, and an extractor function $\textsf{Ext}$  mapping from $\{0,1\}^m$ into $\{0,1\}^n$, such that $Y=\textsf{Ext}(X)$ is (nearly) uniformly distributed in $\{0,1\}^n$. 
To formalize the idea further, we say that a (deterministic) $\epsilon$-extractor $\textsf{Ext}$ is mapping from $\{0,1\}^m$ into $\{0,1\}^n$ such that
the distribution $\P_Y$ of $Y=\textsf{Ext}(X)$ is close to the uniform distribution $\mathbb U_n$ on $\{0,1\}^n$ in the sense that
$$
  d_{\text{TV}}(\P_Y,\mathbb U_n) =\frac 12\sum_{w\in\{0,1\}^n} |\P_Y(w)-2^{-n}|<\epsilon,
$$
where $d_{\text{TV}}$ is the total variation metric. Unfortunately, a simple argument (e.g., \cite{Raz}) shows that it is not possible to construct a universal extractor, capable of  producing 
an output bit, which is $\epsilon$-close to uniform, $\epsilon<1/2$, for all random vectors with $X\in\{0,1\}^m$  with large min-entropy  $\mathbf{H}_{\infty}(X)\ge m-1$.
Indeed, suppose that such an extractor $\operatorname{\textsf{Ext}}:\{0,1\}^m \rightarrow\{0,1\}$ exists. 
Let $$S_0=\left\{x \in\{0,1\}^m: \operatorname{\textsf{Ext}}(x)=0\right\}\text{ and }S_1=\left\{x \in\{0,1\}^m: \operatorname{\textsf{Ext}}(x)=1\right\}.$$
Note also, that since $S_0\cup S_1=\{0,1\}^m$,  one of the sets
$S_0$ and $S_1$  has cardinality at least $ 2^{m-1}$. Suppose for simplicity that  
$|S_0|\ge 2^{m-1}$ and consider a random element $X_0$, which is uniformly distributed on $S_0$. Then $\mathbf{H}_{\infty}\left(X_0\right) \geq m-1$.
However, $Y=\operatorname{\textsf{Ext}}\left(X_0\right)=0$ identically,
and hence $Y$ is not $\epsilon$-close to $\mathbb U_1$.
Fortunately, one can turn to the so-called seeded randomness extractors.

\medskip
A {\em seeded $(k,\varepsilon)$-extractor} is a function 
$$\textsf{Ext}: \,\{0,1\}^m\times\{0,1\}^{d} \rightarrow \{0,1\}^{n}$$ such that for every $(m,k)$-source $X$, the distribution of $Y=\textsf{Ext}\left(X, Z\right)$, where $Z\sim \mathbb{U}_{d}$, is $\varepsilon$-close to $\mathbb{U}_{n}$.  A seeded extractor, if it exists, is able to take an arbitrary sufficiently 
random input $X$ (measured in terms of its min-entropy), and, hopefully, a relatively short uniformly distributed random seed,
to produce a nearly uniformly distributed output.
The principal question is under which conditions on $m,d,k,n,$ and $\varepsilon$,
a seeded extractor exist. There are numerous results of such nature.
Let us recall the following: 

\begin{theorem}[Theorem 1.5, \cite{thmonefive}]
For every constant $\alpha>0$, and all positive integers $n, k$ and all $\varepsilon>0$, there is an explicit construction of a $(k, \varepsilon)$-extractor $\operatorname{ \textsf{Ext}}:\{0,1\}^{m} \times\{0,1\}^{d} \rightarrow\{0,1\}^{n}$ with $d=O(\log m+\log (1 / \varepsilon))$ and $n \geqslant(1-\alpha) k$.
\end{theorem}

Taking into account that the distribution of digits produced by the 
$\beta$-encoder has min-entropy at least of the order of $m\frac{\log\beta_{\min}}{\log 2}$, the above theorem
states that we can produce $n=(1-\alpha) m\frac{\log\beta_{\min}}{\log 2}$ of nearly uniformly distributed binary digits $(a_1,\ldots,a_n)$. Equivalently, we need 
$m=\frac 1{1-\alpha}  n\frac { \log 2}{\log \beta_{\min}}$ output bits of the random $\beta$-encoder to obtain 
$n$ binary well-distributed bits.

\medskip
One can compare this result with the result of Theorem \ref{t:main}, which states that we would need at least
$n\frac{\log 2}{\log \beta}$ output bits of the $\beta$-encoder, while 
the more robust universal randomness extractor would require
$\frac 1{(1-\alpha)}n \frac{\log 2}{\log \beta}$, i.e., only a fixed fraction more.
Thus, the price we have to pay is rather small since the bits are produced by a relatively cheap circuit working at high clock frequency. Therefore,
switching from a specific extraction scheme based on entropy encoding
suggested by Jitsumatsu et al.~\cite{KJ16} to a universal randomness extractor does not constitute a significant limitation.

\medskip
However, the important point we have not yet taken into account is the need
to use a relatively short, but ``purely random",  seed of length $d=O(\log m+\log (1 / \varepsilon))$. 
In practice one does not have access to such sources of ``pure randomness".
Fortunately, weak sources of randomness, such as $\beta$-encoders, can be used
as seeds as well. This brings us to the discussion of 
extractors with weak random seeds. In \cite{Raz} the following definition of two-sources-extractors is given.

\begin{defn} (Two-Sources-Extractor \cite{Raz})
A function $\textsf{Ext}:\{0,1\}^{m_1} \times\{0,1\}^{m_2} \rightarrow\{0,1\}^n$ is an $\left[\left(m_1, k_1\right),\left(m_2, k_2\right) \mapsto n \sim \varepsilon\right]$-two-sources-extractor if for every $\left(m_1, k_1\right)$ source $X_1$ and every independent $\left(m_2, k_2\right)$-source $X_2$, the distribution of the random variable $\textsf{Ext}\left(X_1, X_2\right)$ is $\varepsilon$-close to $\mathbb U_n$ (i.e., the uniform distribution over $\left.\{0,1\}^n\right)$.
% We say in this case that $E$ extracts $m$ bits with probability of error $\gamma$.
\end{defn}

Similarly, one can define source extractors for any number of sources $\ell\ge 2$,
$$\textsf{Ext}:\{0,1\}^{m_1} \times\{0,1\}^{m_2} \times \cdots \times \{0,1\}^{m_\ell} \rightarrow\{0,1\}^n
$$
such that the extractor $\textsf{Ext}(X_1,\ldots,X_\ell)$ is $\epsilon$-close to $\mathbb U_n$
for all independent $(m_1,k_1)$,$\ldots$,$(m_\ell,k_\ell)$-sources  $(X_1,\ldots,X_\ell)$.

\medskip
The theory of multiple source extractors was actively developed in the past 25 years.
It turns out that there is a significant difference between the cases $\ell=2$ and $\ell\ge 3$. The case $\ell=2$ is substantially more complicated.
It is indeed possible to construct good, efficient two-source extractors, say for $m_1=m_2=m$ with the min-entropy of at least $\frac 12 m$.
%%, c.f., \cite{CG98}

\begin{theorem}\cite{Sha}
For every constant $\delta>0$ there is a constant $C>0$ such that for large enough $m$, setting $k=(1 / 2+\delta) m$ and $\epsilon \leq 2^{-\log ^4 m}$ there is an explicit $[(m, k),(m,k) \mapsto n\sim \epsilon]$-two-source extractor $\textsf{Ext}:\{0,1\}^m \times\{0,1\}^m \rightarrow\{0,1\}^n$ for $n=2 k-C \log (1 / \epsilon)$.
\end{theorem}

In our case, given the bound on min-entropy (\ref{eq:lowerminentropy}), that would necessarily imply that we need an extra assumption
that 
$$
\beta_{\min}>\sqrt{2}.
$$
%$$
%\log{\beta_{\min}}/\log 2> 1/2 \iff \log_2\beta_{\min}>1/2\iff \ \beta_{\min}>2^{1/2}=\sqrt{2}
%$$
It is not immediately clear whether such a restriction would constitute a serious limitation 
for applications, but it is clear that such an  a priori assumption would be undesirable.
On the other hand, if one turns to randomness extractors for $\ell$ weak sources with $\ell\ge 3$, assumptions on $\beta_{\min}$ can be relaxed.
Barak, Impagliazzo and Wigderson \cite{BIW} showed using techniques from additive combinatorics that for any $\delta>0$,
there exist randomness extractors requiring only $\ell=\text{poly}(1/\delta)$ independent $(m,\delta m)$ sources, where $\text{poly}$ is some polynomial function.
It means that assuming that $\beta_{\min}>1$, i.e., $\beta_{\min}=1+\tilde\delta$
for some $\tilde\delta>0$ is sufficient.
These results were further improved by   Raz \cite{Raz} who showed that $\ell=3$ is indeed sufficient.

\medskip
The final point of discussion is whether one could 
get $\ell>1$ independent weak sources of randomness. This could be achieved by running the $\beta$-encoder several times, or, running it once, generating a very long series of bits $N\gg 1$, and then extracting strings of length $m$, with sufficiently large gaps between them. 

\section*{Acknowledgments}
We would like to thank Yutaka Jitsumatsu for valuable discussions.

\medskip

\bibliographystyle{alpha} 
\bibliography{lochs} 
%%%\bibliography{crypto}
\end{document}